\documentclass[12pt,a4paper, oneside]{amsart}

\usepackage{amsmath, nicefrac, amsthm, verbatim, amsfonts, mathtools, amssymb, upgreek, xcolor, bbm}
\usepackage{graphics, xspace, enumerate}
\usepackage{stix}
\usepackage{dsfont}
\usepackage[a4paper,margin=2.5cm]{geometry}
\usepackage[bb=boondox]{mathalfa}

\usepackage{graphicx}
\usepackage[colorlinks=true,citecolor=red,urlcolor=blue,linkcolor=red,bookmarksopen=true,unicode=true,pdffitwindow=true]{hyperref}
\usepackage[english]{babel}
\usepackage[languagenames,fixlanguage]{babelbib}
\hypersetup{pdfauthor={}}
\hypersetup{pdftitle={}}

\usepackage{seqsplit,cleveref}

\theoremstyle{plain}
\newtheorem{theorem}{Theorem}[section]

\newtheorem{lemma}[theorem]{Lemma}

\theoremstyle{definition}

\newcommand {\Prob} {\ensuremath{\mathbb{P}}}
\newcommand {\R} {\ensuremath{\mathbb{R}}}

\newcommand {\N} {\ensuremath{\mathbb{N}}}

\newcommand{\chain}[1]{\{#1_n\}_{n\geq0}}

\newcommand{\ttup}[1]{\textup{(}#1\textup{)}}
\newcommand{\df}{\coloneqq}

\newcommand{\E}{\mathrm{e}}

\newcommand{\sfZ}{\mathsf{Z}}

\newcommand{\D}{\mathrm{d}}

\numberwithin{equation}{section}

\title[Learning from non-irreducible Markov chains]{Learning from non-irreducible Markov chains}

\author[N.\ Sandri\'{c}]{Nikola Sandri\'{c}}
\address[Nikola\ Sandri\'{c}]{Department of Mathematics\\University of Zagreb\\ Zagreb\\Croatia}
\email{nikola.sandric@math.hr}

\author[S.\ \v Sebek]{Stjepan\ \v Sebek}
\address[Stjepan\  \v Sebek]{
	Department of Applied Mathematics\\
	Faculty of Electrical Engineering and Computing\\
	University of Zagreb\\ 
	Zagreb\\ 
	Croatia}
\email{stjepan.sebek@fer.hr}

\subjclass[2010]{68W40, 68T10, 60J05}
\keywords{Approximate sample error minimization  algorithm, Generalization bounds, Markov chain, Wasserstein distance}

\begin{document}
\allowdisplaybreaks[4]

\begin{abstract}
	Most of the existing literature on	 supervised machine learning problems focuses on the case when the training data set is drawn from an i.i.d.\ sample. However, many practical  problems   are characterized by temporal  dependence and strong correlation between the marginals of the data-generating process, suggesting that the i.i.d. assumption is not always justified. This problem has been already considered   in the context of Markov chains satisfying the Doeblin condition. This condition, among other things, implies that the chain is not singular in its behavior, i.e.\ it is irreducible. In this article, we  focus on the case when the training data set is drawn from  a not necessarily  irreducible Markov chain. Under the assumption that the   chain is uniformly ergodic with respect to the $\mathrm{L}^1$-Wasserstein distance, and
certain regularity assumptions on the  hypothesis class and the state space of the chain,  
 we first obtain a uniform convergence result for the corresponding sample error, and then
we conclude learnability of the 
approximate sample error minimization  algorithm and find its  generalization bounds. At the end,  a relative uniform convergence result for the  sample error is also discussed.
\end{abstract}

\maketitle

\section{Introduction}\label{S1}

Let  $(\mathsf{X},\mathcal{X})$ and $(\mathsf{Y},\mathcal{Y})$ be measurable spaces, and let $(x_1,y_1),\dots,(x_n,y_n)\in\mathsf{X}\times\mathsf{Y}$ (the so-called training  data set) be drawn from random elements $(X_1,Y_1),\dots,(X_n,Y_n):\Omega\to\mathsf{X}\times\mathsf{Y}$ defined on a probability space $(\Omega,\mathcal{F},\Prob)$. In general,  the probability measure $\Prob$ (or the distribution of the data-generating process $(X_1,Y_1),\dots,(X_n,Y_n)$) is not necessarily known. 
The main task of supervised machine learning is, given a family of measurable functions $\mathscr{H}\subseteq\mathsf{Y}^{\mathsf{X}}$ (the so-called hypothesis class), to construct a learning algorithm $\mathcal{A}:\bigcup_{n=1}^\infty(\mathsf{X}\times\mathsf{Y})^n\to\mathscr{H}$ that accurately  predicts a functional relation  between the input (first coordinate) and output (second coordinate) observable. We remark that even if there is an exact functional relationship, this function is in general not known to us and it does not have to belong to the  class  $\mathscr{H}$. More precisely, given  a  measurable function $\ell:\mathcal{Y}\times\mathcal{Y}\to[0,\infty)$ (the so-called loss function) and $\varepsilon,\delta\in(0,1)$, the aim is to construct an algorithm  $\mathcal{A}$ and find $n(\varepsilon,\delta)\in\N$, such that for any $\Prob$ on $(\Omega,\mathcal{F})$,
$$\Prob\bigl(|\mathrm{er}_\Prob(\mathcal{A}((X_1,Y_1),\dots,(X_n,Y_n))-\inf_{h\in\mathscr{H}}\mathrm{er}_\Prob(h)|<\varepsilon\bigr)\ge 1-\delta\qquad \forall n\ge n(\varepsilon,\delta),$$ where $\mathrm{er}_\Prob(h)\df\mathbb{E}_\Prob[\ell(h(X),Y)]$ is the expected loss of the hypothesis $h\in\mathscr{H}$ with respect to the distribution of data $\Prob$. The smallest $n(\varepsilon,\delta)$ (for given $\varepsilon,\delta\in(0,1)$) for which the above relation holds is called the sample complexity of $\mathcal{A}$. Classical assumption in this problem is  the 
 i.i.d.\ nature of the elements $(X_1,Y_1),\dots,(X_n,Y_n)$, see e.g.\  \cite{Anthony-Bartlett-Book-1999} and \cite{Shalev-Shwartz-Ben-David-Book-2014}. 
 However, many practical supervised machine learning problems  are characterized by temporal  dependence and strong correlation between the marginals of the data-generating process, suggesting that the i.i.d.\ assumption is not always justified.

 In this article, we assume that the training data set is drawn from a temporally-homogeneous Markov chain. More precisely,  let $(\mathsf{Z},\mathcal{Z})$, $\mathsf{Z}\subseteq\mathsf{X}\times\mathsf{Y}$, be  a measurable space, and let $\mathrm{pr}_\mathsf{X}:\mathsf{Z}\to\mathsf{X}$ and $\mathrm{pr}_\mathsf{Y}:\mathsf{Z}\to\mathsf{Y}$ 
be the projection mappings  such that $\mathrm{pr}^{-1}_\mathsf{X}(\mathrm{pr}_\mathsf{X}(\mathsf{Z})\cap\mathcal{X})\subseteq\mathcal{Z}$ and $\mathrm{pr}^{-1}_\mathsf{Y}(\mathrm{pr}_\mathsf{Y}(\mathsf{Z})\cap\mathcal{Y})\subseteq\mathcal{Z}$.
This, in particular, implies that  $\mathrm{pr}_\mathsf{X}$ and $\mathrm{pr}_\mathsf{Y}$ are measurable.  Further, let $(\Omega,\mathcal{F},\Prob_{z},\{Z_n\}_{n\ge0},\chain{\mathcal{F}})_{z\in\mathsf{Z}}$, denoted by $\{Z_n\}_{n\ge0}$ in the sequel,  be a temporally-homogeneous Markov chain on $\mathsf{Z}$, in the sense of \cite{Meyn-Tweedie-Book-2009}. 
The process $\{Z_n\}_{n\ge0}=$ \linebreak $ \{(\mathrm{pr}_\mathsf{X}(Z_n),\mathrm{pr}_\mathsf{Y}(Z_n))\}_{n\ge0}$ is viewed as a sequence of input data (the first component) together with their outputs (the second component). As commented above, the main aim  is to learn a function (from a given hypothesis class $\mathscr{H}$) that, given a training data set  $z_0,\dots,z_{n-1}\in\mathsf{Z}$ drawn from the first $n$ samples of $\{Z_n\}_{n\ge0}$, best approximates the relationship  between the input and output  observable. 
This type of problem has already been considered in the literature, see the literature review part below. However, in all these works a crucial assumption is that the process $\{Z_n\}_{n\ge0}$ is not singular in its behavior, i.e.\ it is irreducible: there is a $\sigma$-finite measure $\upvarphi(\D z)$ on
$\mathcal{Z}$ such that whenever $\upvarphi(B)>0$ we have
$\sum_{n=0}^\infty \Prob_z(Z_n\in B)>0$ for all $z\in\sfZ$. A typical example of this type is given as follows. Let $\{X_n\}_{n\ge0}$ be an irreducible Markov chain on $\mathsf{X}$, and let $f:\mathsf{X}\to\mathsf{Y}$ be measurable. According to \Cref{LM}, the process $Z_n=(X_n,f(X_n))$, $n\ge0$, is an irreducible Markov chain on $\mathsf{Z}=\{(x,f(x)):x\in\mathsf{X}\}$. The first component can be understood as a state of the system (vector of cepstral coefficients in speech recognition problems, vector of the position and velocity of the center of gravity of a moving object in object tracking problems, category of a unit-time price change in market prediction problems). The second component represents the label of the corresponding state (e.g., emotional state of the speaker, temporal distance of the tracking object from a referent point, trading activity: buy/sell/wait). Irreducibility here means that the system can move from one state to any other, in a finite number of time steps. Typically, the labeling function $f(x)$ is unknown. Hence, the goal is, basing upon a given training data set and hypothesis class $\mathscr{H}$ (which does not necessarily contain $f(x)$), to find (by constructing an appropriate learning algorithm) a hypothesis $h\in\mathscr{H}$ which best approximates the true labeling function $f(x)$.

In this article, we study the learning problem in the case when $\{Z_n\}_{n\ge0}$ is not necessarily irreducible.
We follow the standard  statistical learning theory scheme: 
 we first obtain a uniform convergence result for the corresponding sample error, and then we conclude learnability of the 
  approximate sample error minimization  algorithm and find its  generalization bounds.

\subsection*{Main results} We start with the description of  the model. Assume the following: 

\smallskip

\begin{description}
	\item [(A1)]  $\mathsf{Z}$ is  a complete separable metric space with bounded metric $\uprho$, and $\mathcal{Z}$ is the corresponding Borel $\sigma$-algebra $\mathfrak{B}(\sfZ)$.
	
	\smallskip
	
	\item[(A2)]   $\mathsf{Y}\subseteq\R$,  $\mathcal{Y}=\mathfrak{B}(\R)\cap\mathsf{Y}$ (the standard  relative Euclidean $\sigma$-algebra), and $\mathscr{H}$   is a totally bounded metric space, i.e.\ $\mathscr{H}$ admits a metric $\uprho_\mathscr{H}$ such that for each $\epsilon>0$ there is finite $\mathscr{H}_\epsilon\subseteq\mathscr{H}$ with the property that for every $h\in\mathscr{H}$ there is $h_\epsilon\in\mathscr{H}_\epsilon$ with $\uprho_\mathscr{H}(h,h_\epsilon)<\epsilon$.
\end{description}

\smallskip

\noindent For $\epsilon>0$  let  $$\mathcal{N}(\epsilon,\mathscr{H},\uprho_\mathscr{H})\df\min\{|\mathscr{H}_\epsilon|\colon \mathscr{H}_\epsilon\text{ is an } \epsilon\text{-covering of } \mathscr{H} \text{ in the sense of (A2)}\}$$ be the $\epsilon$-covering number of $\mathscr{H}$. Typical examples of classes of functions satisfying (A2) are the following:
\begin{itemize}
	\item [(i)]  Let $\mathsf{X}=[0,T]$, let $\mathsf{Y}=[-V/2,V/2]$ and let $\mathscr{H}$ be a class of functions with total variation at most $V$, for some $T,V>0$. Further, let $\uprho_\mathscr{H}$ be the corresponding $\mathrm{L}^1$-metric (with respect to the Lebesgue measure). Then, in \cite[Theorem 1]{Bartlett-Kulkarni-Posner-1997} it has been shown that $$\mathcal{N}(\epsilon,\mathscr{H},\uprho_\mathscr{H})\le 2^{13VT/\epsilon}.$$
	\item [(ii)] Let $\mathsf{X}\subset\R^d$ be compact and  let $\mathscr{H}\subset\mathcal{C}^\gamma(\mathsf{X},\R)$   be bounded (with respect to the corresponding H\"{o}lder metric) for some $0<\gamma\le1$. 
	Then, for $\uprho_\mathscr{H}$ being the H\"{o}lder or sup-metric, in \cite[Theorem 4]{Zhou-2003} it has been shown that there is $C>0$  such that \begin{equation}\label{eq:ex}\mathcal{N}(\epsilon,\mathscr{H},\uprho_\mathscr{H})\le\E^{C\epsilon^{-2d/\gamma}}.\end{equation}
\end{itemize}

Further, for each $h\in\mathscr{H}$ define $\ell_h:\mathsf{Z}\to[0,\infty)$ by $$\ell_h(z)\df\ell\bigl(h(\mathrm{pr}_\mathsf{X}(z)), \mathrm{pr}_\mathsf{Y}(z)\bigr).$$ 
We next assume:

\medskip

\begin{description}
	\item[(A3)]  there are $L,\bar L\ge0$ such that $\displaystyle|\ell_{h_1}(z_1)-\ell_{h_2}(z_2)|\le L\uprho(z_1,z_2)+\bar L \uprho_\mathscr{H}(h_1,h_2)$ for all $z_1,z_2\in\sfZ$ and $h_1,h_2\in\mathscr{H}$.
\end{description}

\medskip

\noindent Examples of classes of functions for which (A3) holds are the following:
\begin{itemize}
	\item [(i)] Let $\mathsf{X}$ and $\mathscr{H}$ be as in (ii) above, with $\gamma=1$ (the class of Lipschitz continuous functions). Assume also that $\mathsf{Y}$ is bounded. Then, (A3) holds with $\ell(y,\bar y)=|y-\bar y|^2$, and $\uprho$ and  $\uprho_\mathscr{H}$ being the  Euclidean and sup metric, respectively.
	\item[(ii)] Let $\mathsf{X}$ and $\mathsf{Y}$ be as in (i), and let $\mathscr{H}\subseteq\{h\in\mathsf{Y}^\mathsf{X}\colon h \text{ is Lipschitz continuous and } h(x_0)=y_0\}$ for some $x_0\in\mathsf{X}$ and  $y_0\in\mathsf{Y}$. Then, (A3) holds with $\ell(y,\bar y)=|y-\bar y|^2$, and $\uprho$ and  $\uprho_\mathscr{H}$ being the Euclidean and Lipschitz metric, respectively.
\end{itemize}

Denote by $\mathcal{P}(z,\D \bar{z})$ the (one-step) transition kernel of $\{Z_n\}_{n\ge0}$, and let 
$\mathscr{P}_1(\sfZ)$ be the class of all probability measures on $\mathsf{Z}$ having finite first moment.
Recall that the
$\mathrm{L}^1$-Wasserstein distance on $\mathscr{P}_1(\sfZ)$  is defined by
\begin{equation*}
\mathscr{W}(\upmu_1,\upmu_2)\df\inf_{\Pi\in\mathcal{C}(\upmu_1,\upmu_2)}
 \int_{\sfZ\times\sfZ}\uprho(z,\bar z)\,
\Pi(\D{z},\D{\bar z}),
\end{equation*}
where $\mathcal{C}(\upmu_1,\upmu_2)$ is the family of couplings of
$\upmu_1(\D z)$ and $\upmu_2(\D z)$,
i.e. $\Pi\in\mathcal{C}(\upmu_1,\upmu_2)$ if, and only if, $\Pi(\D z, \D \bar z)$
is a probability
measure on $\sfZ\times\sfZ$ having $\upmu_1(\D z)$ and $\upmu_2(\D z)$ as its marginals.
We finally
assume:

\medskip

\begin{description}
	\item[(A4)]  $\mathsf{Y}$ is bounded and there is $\eta\in(0,1)$ such that $$\mathscr{W}\bigl(\mathcal{P}(z_1,\D \bar{z}),\mathcal{P}(z_2,\D \bar z)\bigr)\le(1-\eta)\,\uprho(z_1,z_2).$$
\end{description}

\medskip

\noindent An example satisfying (A4) is as follows.
Let $\{\chi_n\}_{n\geq1}$ be a sequence of i.i.d. $\R$-valued random variables defined on a probability space $(\Omega,\mathcal{F},\mathbb{P})$, satisfying $\mathbb{P}(\chi_n=0)=\mathbb{P}(\chi_n=1/2)=1/2$. Define $$X_{n+1}\df\frac{1}{2}X_n+\chi_{n+1}$$ with $X_0=x\in[0,1]$.
Clearly, $\chain{X}$ is a Markov chain on  $\mathsf{X}=[0,1]$ with transition kernel $\mathcal{P}_X(x,\D\bar x)=\mathbb{P}(\chi_1+x/2\in \D \bar x)$. 
Further, let $f:\mathsf{X}\to\mathsf{Y}=\R$ be Lipschitz continuous with Lipschitz constant $\mathrm{Lip}(f)<\sqrt{3}$ (hence, $f(x)$ is of bounded variation), and set $Z_n\df (X_n,f(X_n))$ for $n\ge0$. Then, $\chain{Z}$ is a Markov chain on $\sfZ=\{(x,f(x))\colon x\in\mathsf{X}\}$ with transition kernel $\mathcal{P}(z,\D\bar z)=\mathbb{P}(\chi_1+\mathrm{pr}_\mathsf{X}(z)/2\in\D\mathrm{pr}_\mathsf{X}(\bar z))$, see \Cref{LM}.
For $\uprho$ we take the standard Euclidean metric. 
Clearly, $\chain{Z}$ is not irreducible. Namely, for $z\in\sfZ$ such that  $\mathrm{pr}_\mathsf{X}(z)\in\mathbb{Q}$ it holds that $\mathcal{P}^n(z,\sfZ\cap(\mathbb{Q}^c\times\R))=0$ for all $n\ge1$, and analogously for $z\in\sfZ$ such that  $\mathrm{pr}_\mathsf{X}(z)\in\mathbb{Q}^c$ it holds that $\mathcal{P}^n(z,\sfZ\cap(\mathbb{Q}\times\R))=0$ for all $n\ge1$.
Further, let ${\rm Length}(\D z)$ be the arc-length measure on $\sfZ$. It is then easy to see that 
$\uppi(\D z)\df {\rm Length}(\D z)/{\rm Length}(\sfZ)$ satisfies $\int_\sfZ \mathcal{P}( z,\D\bar z)\uppi(\D z)=\uppi(\D\bar z),$ i.e.\ $\uppi(\D z)$ is an invariant probability measure for $\{Z_n\}_{n\ge0}$. Here,  ${\rm Length}(\sfZ)$ is the length of  $\sfZ$.
Observe that $\uppi(\D \bar z)$ and $\mathcal{P}(z,\bar z)$ are mutually  singular, implying that the relation in \cref{eq:TV1} cannot hold. Finally,  we show the relation in (A4). By  Kantorovich-Rubinstein theorem we have that
$$
\mathscr{W}(\upmu_1,\upmu_2) = \sup_{\{g\colon\mathrm{Lip}(g)\le1\}}\,
| \upmu_1(g)-\upmu_2(g)|,
$$
where the supremum is taken over all Lipschitz continuous functions
$g\colon\sfZ\to\R$ with Lipschitz constant $\mathrm{Lip}(g)\le1$ and, for a probability measure $\upmu(\D z)$ on $\sfZ$ and a measurable function $f:\sfZ\to\R$, the symbol $\upmu(f)$ stands for $\int_\sfZ f(z)\upmu(\D z)$, whenever the integral is well defined.
Thus, for any $z_1=(x_1,f(x_1)),z_2=(x_2,f(x_2))\in\sfZ$ it holds that
\begin{align*}
&\mathscr{W}(\mathcal{P}(z_1,\bar z),\mathcal{P}(z_2,\bar z)) \\
&= \sup_{\{g\colon\mathrm{Lip}(g)\le1\}}\,
\left|\int_{\mathsf{Z}} g(\bar z) \mathbb{P}\bigl(\chi_1+\mathrm{pr}_\mathsf{X}(z_1)/2\in\D\mathrm{pr}_\mathsf{X}(\bar z)\bigr)-\int_{\mathsf{Z}} g(\bar z) \mathbb{P}\bigl(\chi_1+\mathrm{pr}_\mathsf{X}(z_2)/2\in\D\mathrm{pr}_\mathsf{X}(\bar z)\bigr)\right|\\
&=\sup_{\{g\colon\mathrm{Lip}(g)\le1\}}\,
\Bigg|\frac{1}{2}\left(g\left(x_1/2,f\left(x_1/2\right)\right) -g\left(x_2/2,f\left(x_2/2\right)\right)\right)\\&\ \ \ +\frac{1}{2}\left(g\left((x_1+1)/2,f\left((x_1+1)/2\right)\right)-g\left((x_2+1)/2,f\left((x_2+1)/2)\right)\right)\right)\Bigg|\\
&\le \frac{1}{2}\uprho\left(\left(x_1/2,f\left(x_1/2\right)\right),\left(x_2/2,f\left(x_2/2\right)\right)\right)\\&\ \ \ +\frac{1}{2}\uprho\left(\left((x_1+1)/2,f\left((x_1+1)/2\right)\right),\left((x_2+1)/2,f\left((x_2+1)/2\right)\right)\right)\\
&\le \frac{\sqrt{1+\mathrm{Lip}(f)^2}}{2}|x_1-x_2|\\
&\le \frac{\sqrt{1+\mathrm{Lip}(f)^2}}{2} \uprho(z_1,z_2),
\end{align*}
which proves (A4) with $\eta=1-\sqrt{1+\mathrm{Lip}(f)^2}/2$ (recall that $\mathrm{Lip}(f)<\sqrt{3}$).

According to \cite[Theorem 2.1]{Butkovsky-2014}  (by taking $V\equiv1$, $\varphi(t)=t$ and $K\ge 1$) the relation in (A4) implies that 
\begin{itemize}
	\item [(i)] $\{Z_n\}_{n\ge0}$ admits a unique invariant probability measure $\uppi(\D z)$;
	\item [(ii)]  there are $C_1,C_2>0$ such that \begin{equation}\label{eq1.1}\mathscr{W}\bigl(\mathcal{P}^n( z,\D\bar z),\uppi(\D \bar z)\bigr)\le  C_1\E^{-C_2n}.\end{equation}
\end{itemize}

Before stating the main results of this article, we introduce some notation we need. For $n\in\N$, $z\in\sfZ$ and $h\in\mathscr{H}$, let $\hat{\rm er}_n(h)\df \frac{1}{n}\sum_{i=0}^{n-1}\ell_h(Z_i)$, ${\rm er}_\uppi(h)\df  \uppi(\ell_h)$, $\ell_h^\uppi(z)\df \ell_h(z)-{\rm er}_\uppi(h)$ and ${\rm opt}_\uppi(\mathscr{H})\df\inf_{h\in\mathscr{H}}{\rm er}_\uppi(h)$. Observe that, according to (A3) and boundedness of $\uprho$, ${\rm er}_\uppi(h)$ and $\ell_h^\uppi(z)$ are well defined.
Next, for 
 $\varepsilon>0$ the $\varepsilon$-approximate sample error minimization ($\varepsilon$-ASEM) algorithm for $\mathscr{H}$ is defined  as a mapping $\mathcal{A}^\varepsilon:\bigcup_{n=1}^\infty\sfZ^n\to \mathscr{H}$ satisfying $$\frac{1}{n} \sum_{i=0}^{n-1}\ell_{\mathcal{A}^\varepsilon(z_0,\dots,z_{n-1})}(z_i)<\inf_{h\in\mathscr{H}} \frac{1}{n} \sum_{i=0}^{n-1}\ell_{h}(z_i)+\varepsilon.$$
\begin{theorem}\label{TM1.1} Assume (A1)-(A4), and fix $\varepsilon,\delta\in(0,1)$. Then, for any initial distribution $\upmu(\D z)$ of $\{Z_n\}_{n\ge0}$,
	$$\Prob_\upmu\left(|\mathrm{er}_\uppi\bigl(\mathcal{A}^\varepsilon(Z_0,\dots,Z_{n-1})\bigr)-\mathrm{opt}_\uppi(\mathscr{H})|<5\varepsilon\right)\ge 1-\delta \qquad \forall n\ge n_1(\varepsilon,\delta),$$
	where $$n_1(\varepsilon,\delta)\df \max\left\{\frac{16C_1L}{\varepsilon(1-\E^{-C_2})},\frac{128C_1^2L^2\ln\frac{\mathcal{N}(\varepsilon/4\bar L,\mathscr{H},\mathsf{d}_\mathscr{H})}{\delta}}{\varepsilon^2(1-\E^{-C_2})^2}\right\}.$$
	\end{theorem}

\noindent Notice that in the case when $\mathcal{N}(\varepsilon, \mathscr{H}, \mathsf{d}_{\mathscr{H}}) \le \E^{C\varepsilon^{-2d / \gamma}}$ (see \cref{eq:ex}) the sample size in \Cref{TM1.1}  is of order $\varepsilon^{-(2 + 2d / \gamma)}$. 
If we additionally assume  that $\chain{Z}$ satisfies  Doeblin condition: there are a probability measure $\upphi(\D z)$ on $\mathsf{Z}$, $\alpha,\beta\in(0,1)$ and $n\in\N$, such that whenever $\upphi(B)>\alpha$,
$$\inf_{z\in\mathsf{Z}}\mathcal{P}^n(z,B)>\beta,$$ 
in \cite[(10)]{Zou-Zhang-Xu-2009} it has been shown that the sample size   is again of order $\varepsilon^{-(2 + 2d / \gamma)}$. Recall that under (A1) the Doeblin condition implies \cref{eq1.1}, and hence the results from \Cref{TM1.1} (by additionally assuming (A2)-(A3)), see \cite[Theorem 16.0.2]{Meyn-Tweedie-Book-2009} and \cite[Theorem 6.15]{Villani-Book-2009}.
On the other hand, if we relax the Doeblin condition and assume that $\chain{Z}$ is $\mathcal{V}$-ergodic (see \cref{eq:TV1}), in \cite[Proposition 1]{Zou-Xu-Chang-2012} it has been shown that the sample size has higher order: $\varepsilon^{-(4 + 4d / \gamma)}$. 
Let us remark that   Doeblin condition implies $\mathcal{V}$-ergodicity (see  \cite[Theorem 16.0.2]{Meyn-Tweedie-Book-2009}), which in turn implies irreducibility. Generally speaking, this means that irreducibility property itself neither positively nor negatively affects the sample size order.

The main step in the proof of \Cref{TM1.1} is the following sample error uniform convergence result 
$$\Prob_z\left(\sup_{h\in\mathscr{H}}|\hat{\mathrm{er}}_n(h)-\mathrm{er}_\uppi(h)|>\varepsilon\right) \le \mathcal{N}(\varepsilon/(4\bar L),\mathscr{H},\mathsf{d}_\mathscr{H})\E^{-\bigl(\frac{\varepsilon n(1-\E^{-C_2})}{4C_1L}-2\bigr)^2/(2n)}\quad   \forall n\ge \frac{8C_1L}{\varepsilon(1-\E^{-C_2})},$$
(see \cref{eq:unif}).  
In the case when $\mathcal{N}(\varepsilon, \mathscr{H}, \mathsf{d}_{\mathscr{H}}) \le \E^{C\varepsilon^{-2d / \gamma}}$ the sample size in the above relation  is again of order $\varepsilon^{-(2 + 2d / \gamma)}$. On the other hand, one expects that the sample size  needed to ensure that $(1 + \alpha)\hat{\rm er}_n(h) + \varepsilon \ge \mathrm{er}_{\uppi}(h)$ for some $\alpha, \varepsilon > 0$ is of much smaller order. This problem 
was  studied in \cite[Subsection 5.5]{Anthony-Bartlett-Book-1999}.  Crucial step in the proof of this result  is to  study the uniform relative error \begin{equation*}
	\sup_{h \in \mathscr{H}} \frac{\left| \mathrm{er}_{\uppi}(h) - \hat{\rm er}_n(h) \right|}{\sqrt{\mathrm{er}_{\uppi}(h)}},
\end{equation*} instead of the uniform standard  error
\begin{equation*}
	\sup_{h \in \mathscr{H}} \left| \mathrm{er}_{\uppi}(h) - \hat{\rm er}_n(h) \right|.
\end{equation*}
Studying this quantity is additionally motivated by the observation that a uniform convergence bound fails to capture the phenomenon that for those hypotheses $h \in \mathscr{H}$ for which the true error $\mathrm{er}_{\uppi}(h)$ is small, the deviation $\mathrm{er}_{\uppi}(h) - \hat{\rm er}_n(h)$ is also small with large probability. However, the relative deviation $(\mathrm{er}_{\uppi}(h) - \hat{\rm er}_n(h)) / \sqrt{\mathrm{er}_{\uppi}(h)}$ is more appropriate for capturing this phenomenon (see \cite[Section 4]{Zou-Zhang-Xu-2009}). With relative error we penalize the difference between the sample error and the true error much more when the true error is very small in the first place. Further, notice that studying relative error makes
sense only when the true error is bounded away from zero (see  \cref{eq:mM} below),  since relative deviation $(\hat{\mathrm{er}}_{\uppi}(f)-\mathrm{er}_{\uppi}(f))  / \sqrt{\mathrm{er}_{\uppi}(f)}$ is not well defined for the true function $f(x)$.  Even though this assumption does seem a bit unnatural, in many practical applications there will not even be a true function in the background, but some distribution, and even if there will be a functional relationship between the input and the output observable, this relationship can often be pretty complex and hence not contained in the class $\mathscr{H}$. 

\begin{theorem}\label{TM1.2}
Assume (A1)-(A4) and assume additionally that there exists a constant $m> 0$ such that 
\begin{equation}\label{eq:mM}
	 \inf_{h\in\mathscr{H}}\mathrm{er}_{\uppi}(h)\ge m.
\end{equation} 
Fix $\alpha>0$ and $\varepsilon,\delta\in(0,1)$. Then, for any initial distribution $\upmu(\D z)$ of $\{Z_n\}_{n\ge0}$,
\begin{equation*}
\mathbb{P}_{\upmu}\left( \exists\, h \in \mathscr{H} : \mathrm{er}_{\uppi}(h) > (1 + \alpha) \hat{\rm er}_n(h) + \varepsilon \right) \le \delta \qquad \forall n \ge n_2(\varepsilon, \delta),
\end{equation*}
where
\begin{align*}
n_2(\varepsilon,\delta)
& \df \max\left\{ \frac{8C_1 L\sqrt{1+1/\alpha}}{\sqrt{m\varepsilon} (1 - \E^{-C_2})}, \right. \\
& \qquad \qquad \left. \frac{32C_1^2L^2(1+1/\alpha)}{m\varepsilon(1-\E^{-C_2})^2} \left( \frac{\sqrt{m\varepsilon}(1-\E^{-C_2})}{C_1 L\sqrt{1+1/\alpha}} + \ln \left( \frac{\mathcal{N}\left( \frac{\sqrt{m\varepsilon}}{4\bar{L} \sqrt{1 + 1/\alpha}}, \mathscr{H}, \mathsf{d}_{\mathscr{H}} \right)}{\delta} \right) \right) \right\}.
\end{align*}
	\end{theorem}

\noindent 
Observe that in the case when $\mathcal{N}(\varepsilon, \mathscr{H}, \mathsf{d}_{\mathscr{H}}) \le \E^{C \varepsilon^{-2d / \gamma}}$ the sample size in \Cref{TM1.2} is  of order $\varepsilon^{-(1 + d / \gamma)}$, which is  smaller than $\varepsilon^{-(2 + 2d / \gamma)}$, as conjectured.


\subsection{Literature review}  
Our work contributes to the understanding of  statistical properties of supervised learning problems. Most of the existing literature focuses on the case when the training data set is drawn from an i.i.d.\ sample, see the classical monographs \cite{Anthony-Bartlett-Book-1999} and \cite{Shalev-Shwartz-Ben-David-Book-2014}. However, many practical supervised learning problems (e.g.\ speech recognition, object tracking, market prediction, etc.) are characterized by temporal  dependence and strong correlation between the marginals of the data-generating process, suggesting that the i.i.d.\ assumption  is not always justified. 
This problem has been first addressed  in \cite{Yu-1994}, \cite{Meir-2000} and \cite{Vidyasagar-Book-2003} in the context of stationary $\beta$-mixing sequences, and in \cite{Modha-Masry-1996} in the case of  stationary $m$-dependent $\alpha$-mixing sequences. In  \cite{Steinwart-Hush-Scovel-2009} the authors relax the stationarity assumption and require that the
  data-generating process satisfies a certain law of large numbers only. This, for example, includes (not necessarily stationary) $\alpha$-mixing sequences and Markov chains satisfying the Doeblin condition. Recall that the Doeblin condition implies irreducibility, aperiodicity and uniform exponential ergodicity of the chain  (see \cite[Theorem 16.0.2]{Meyn-Tweedie-Book-2009}). The learning problem of the later model was further investigated in a series of articles \cite{Gamarnik-2003}, \cite{Zou-Zhang-Xu-2009}, \cite{Zou-Li-Xu-2012}, \cite{Zou-Li-Xu-Luo-Tang-2013} and \cite{Zou-Xu-Xu-2014}.
  We remark that a crucial step in all these works is the Hoeffding's inequality for uniformly ergodic Markov chains, obtained in \cite{Glynn-Ormoneit-2002}. 
  Generalization of these results to irreducible aperiodic   $\mathcal{V}$-ergodic Markov chains has been obtained in \cite{Zou-Xu-Xu-2014}. Actually  the proof of the main results of this article (Theorems 1 and 2) is taken from \cite[Theorem 3.5]{Vidyasagar-Book-2003}. We also remark that these results do not give generalization bounds in the full sense. Namely, a crucial assumption in the article is  that the  data-generating process starts from its invariant measure, which is unrealistic. A way to partly fix this gap is presented below.

Recall that  the chain $\{Z_n\}_{n\ge0}$ is said to be
\begin{enumerate}
	\item [\ttup{i}]
	aperiodic if
	there does not exist a partition $\{B_1,\dots,B_k\}\subseteq\mathfrak{B}(\sfZ)$ with
	$k\ge2$ of $\sfZ$ such that $\mathcal{P}(z,B_{i+1}) = 1$ for all $z\in B_i$ and all
	$1 \le i \le k -1$, and $\mathcal{P}(z, B_1) = 1$ for all $z \in B_k$.
\item [\ttup{ii}] $\mathcal{V}$-ergodic if it admits an invariant probability measure $\uppi(\D z)$ and there are $\mathcal{V}:\mathsf{Z}\to[1,\infty)$ with $\int_\sfZ\mathcal{V}(z)\uppi(\D z)<\infty$, $C>0$ and $\kappa\in(0,1)$, such that \begin{equation}\label{eq:TV1}\lVert\mathcal{P}^n(z,\D\bar z)-\uppi(\D\bar z)\rVert_\mathcal{V}\le C\mathcal{V}(z)\kappa^n\end{equation} for all $z\in\mathsf{Z}$ and $n\ge0$.  
\end{enumerate}
Here, for a signed measure $\upmu(\D\mathsf{z})$ on $\mathsf{Z}$, $$ \lVert\upmu(\D z)\rVert_\mathcal{V}\df \sup_{|f|\le\mathcal{V}}|\upmu(f)|.$$
Further, the $\beta$-mixing  (or complete regularity, or the Kolmogorov) coefficient of the  chain $\{Z_n\}_{n\ge0}$ with initial distribution $\upmu(\D z)$ is defined as $$\beta^\upmu(n)\df\sup_{m\ge0}\mathbb{E}_\mu\left[\sup_{B\in\sigma\{Z_k: k\ge m+n\}}|\mathbb{P}_\upmu(B|\mathcal{F}_m)-\mathbb{P}_\upmu(B)|\right].$$  If $\uppi(\D z)$ is an invariant probability measure of $\{Z_n\}_{n\ge0}$, then, by employing the Markov property  and stationarity, 
$$\beta^\uppi(n)=\int_{\mathsf{Z}}\lVert\mathcal{P}^n(z,\D \bar z)-\uppi(\D \bar z)\rVert_\mathrm{TV}\uppi(\D  z).$$
In particular, if $\{Z_n\}_{n\ge0}$ is $\upvarphi$-irreducible, aperiodic and $\mathcal{V}$-ergodic with invariant probability measure $\uppi(\D z)$, then $$\beta^\uppi(n)\le\int_{\mathsf{Z}}\lVert\mathcal{P}^n(z,\D \bar z)-\uppi(\D \bar z)\rVert_\mathcal{V}\uppi(\D z)\le C\uppi(\mathcal{V})\kappa^n$$ for all $n\ge0.$ 
Under this assumptions, (A1)-(A2) and

\medskip

\begin{description}
	\item[\textbf{($\overline {\text{A3}})$}] $\displaystyle B\df \sup_{z\in\sfZ,\, h\in\mathscr{H}}\ell_h(z)<\infty$ and  $\displaystyle|\ell_{h_1}(z)-\ell_{h_2}(z)|\le\bar  L \uprho_\mathscr{H}(h_1,h_2)$ for some $\bar L>0$, and  all $z\in\sfZ$ and $h_1,h_2\in\mathscr{H}$
\end{description}

\medskip

\noindent  in \cite[Theorem 2]{Zou-Xu-Chang-2012} (as we have already commented, the proof of this result has been taken from \cite[Theorem 3.5]{Vidyasagar-Book-2003}) it has been shown that
for fixed $\varepsilon\in(0,3B)$,
\begin{align*}\Prob_\uppi\left(\sup_{h\in\mathscr{H}}|\hat{\mathrm{er}}_n(h)-\mathrm{er}_\uppi(h)|>\varepsilon\right)&\le 2(1+C\uppi(\mathcal{V})\E^{-2})\mathcal{N}(\varepsilon/4 \bar L,\mathscr{H},\uprho_\mathscr{H})\E^{-\frac{\sqrt{n}\sqrt{\ln(\kappa^{-1})}\varepsilon^2}{16\sqrt{2}B^2}}\end{align*} for all $$n\ge n_0(\varepsilon)\df\max\left\{\frac{16+8\sqrt{4+\ln(\kappa^{-1})^2}+8\ln(\kappa^{-1})^2}{\ln(\kappa^{-1})^2},\frac{\ln(\kappa^{-1})}{8}\right\}.$$
From this it follows (see the proof of \Cref{TM1.1}) that  for fixed $\varepsilon>0$ and $\delta\in(0,1)$, and an $\varepsilon$-ASEM algorithm  $\mathcal{A}^\varepsilon$  for $\mathscr{H}$, 
we have that
\begin{equation}\label{eq:TV2}\Prob_\uppi\left(|\mathrm{er}_\uppi\bigl(\mathcal{A}^\varepsilon(Z_0,\dots,Z_{n-1})\bigr)-\mathrm{opt}_\uppi(\mathscr{H})|<5\varepsilon\right)\ge 1-\delta\qquad \forall\, n\ge n_0(\varepsilon,\delta),\end{equation} where  $$n_0(\varepsilon,\delta)\df\max\left\{n_0(\varepsilon),\frac{512 B^4}{\ln(\kappa^{-1})\varepsilon^4}\ln\left(\frac{2(1+C\uppi(\mathcal{V})\E^{-2})\mathcal{N}(\varepsilon/4 \bar L,\mathscr{H},\uprho_\mathscr{H})}{\delta}\right)^2\right\}.$$
However, as we have already commented, this result  is restrictive and impractical since it assumes that the process is initially distributed according to the invariant distribution $\uppi(\D z)$ which is in general unknown to us. It is more natural, and in the spirit of generalization bounds, that the relation in \cref{eq:TV2} holds for every initial distribution of $\{Z_n\}_{n\ge0}$. This result can be slightly generalized by employing (i) H\"{o}lder inequality, or (ii) ergodicity property in \cref{eq:TV1}.

\begin{itemize}
	\item [(i)] Let $\upmu(\D z)$ be any initial distribution of $\{Z_n\}_{n\ge0}$ such that $\upmu\ll\uppi$. Assume also that there is $p\in(1,\infty]$ such that $ \frac{\D\upmu}{\D\uppi}\in\mathrm{L}^p(\uppi)$. Then,
	\begin{align*}&\Prob_\upmu\left(\sup_{h\in\mathscr{H}}|\hat{\mathrm{er}}_n(h)-\mathrm{er}_\uppi(h)|>\varepsilon\right)\\
	&=\int_{\mathsf{Z}}\Prob_z\left(\sup_{h\in\mathscr{H}}|\hat{\mathrm{er}}_n(h)-\mathrm{er}_\uppi(h)|>\varepsilon\right)\upmu(\D z)\\
	&=\int_{\mathsf{Z}}\Prob_z\left(\sup_{h\in\mathscr{H}}|\hat{\mathrm{er}}_n(h)-\mathrm{er}_\uppi(h)|>\varepsilon\right)\frac{\D\upmu}{\D\uppi}(z)\uppi(\D z)\\
	&\le \bigl\lVert\frac{\D\upmu}{\D\uppi}\bigr\rVert_p\left(\int_{\mathsf{Z}}\Prob_z\left(\sup_{h\in\mathscr{H}}|\hat{\mathrm{er}}_n(h)-\mathrm{er}_\uppi(h)|>\varepsilon\right)^q\uppi(\D z)\right)^{1/q}\\
	&\le \bigl\lVert\frac{\D\upmu}{\D\uppi}\bigr\rVert_p\Prob_\uppi\left(\sup_{h\in\mathscr{H}}|\hat{\mathrm{er}}_n(h)-\mathrm{er}_\uppi(h)|>\varepsilon\right)^{1/q},\end{align*} where $q\in[1,\infty)$ is such that $1/p+1/q=1$. Thus,
	for fixed $\varepsilon>0$ and $\delta\in(0,1)$, and an $\varepsilon$-ASEM algorithm  $\mathcal{A}^\varepsilon$  for $\mathscr{H}$, 
	we have that
	$$\Prob_\upmu\left(|\mathrm{er}_\uppi\bigl(\mathcal{A}^\varepsilon(Z_0,\dots,Z_{n-1})\bigr)-\mathrm{opt}_\uppi(\mathscr{H})|<5\varepsilon\right)\ge 1-\delta$$ for all $$n\ge\max\left\{n_0(\varepsilon,\delta),\frac{512 B^4q^2}{\ln(\kappa^{-1})\varepsilon^4}\ln\left(\frac{\bigl(2(1+C\uppi(\mathcal{V})\E^{-2})\mathcal{N}(\varepsilon/4 \bar L,\mathscr{H},\uprho_\mathscr{H})\bigr)^{1/q}\bigl\lVert\frac{\D\upmu}{\D\uppi}\bigr\rVert_p}{\delta}\right)^2\right\} .$$
	
	\item [(ii)] Let now $\upmu(\D z)$ be any initial distribution of $\{Z_n\}_{n\ge0}$ such that $\upmu(\mathcal{V})<\infty$. Fix $\varepsilon>0$ and $\delta\in(0,1)$, and let $\mathcal{A}^\varepsilon$  be an $\varepsilon$-ASEM algorithm    for $\mathscr{H}$. Then, for any $n\ge n_0(\varepsilon,\delta)$ the relation in \cref{eq:TV2} implies
	\begin{align*}
	&\Prob_\upmu\left(|\mathrm{er}_\uppi\bigl(\mathcal{A}^\varepsilon(Z_n,\dots,Z_{2n-1})\bigr)-\mathrm{opt}_\uppi(\mathscr{H})|\ge5\varepsilon\right)\\
	&\le \Prob_\uppi\left(|\mathrm{er}_\uppi\bigl(\mathcal{A}^\varepsilon(Z_0,\dots,Z_{n-1})\bigr)-\mathrm{opt}_\uppi(\mathscr{H})|\ge5\varepsilon\right)\\&\ \ \ +\bigl|\Prob_\upmu\left(|\mathrm{er}_\uppi\bigl(\mathcal{A}^\varepsilon(Z_n,\dots,Z_{2n-1})\bigr)-\mathrm{opt}_\uppi(\mathscr{H})|\ge5\varepsilon\right)\\&\hspace{1cm}-\Prob_\uppi\left(|\mathrm{er}_\uppi\bigl(\mathcal{A}^\varepsilon(Z_0,\dots,Z_{n-1})\bigr)-\mathrm{opt}_\uppi(\mathscr{H})|\ge5\varepsilon\right)\bigr|\\
	&<\delta+\big|\mathbb{E}_\upmu\bigl[\Prob_{Z_n}\left(|\mathrm{er}_\uppi\bigl(\mathcal{A}^\varepsilon(Z_0,\dots,Z_{n-1})\bigr)-\mathrm{opt}_\uppi(\mathscr{H})|\ge5\varepsilon\right)\bigr]\\&\hspace{1.3cm} -\mathbb{E}_\uppi\bigl[\Prob_{Z_n}\left(|\mathrm{er}_\uppi\bigl(\mathcal{A}^\varepsilon(Z_0,\dots,Z_{n-1})\bigr)-\mathrm{opt}_\uppi(\mathscr{H})|\ge5\varepsilon\right)\bigr]\big|\\
	&\le\delta+\left|\int_{\mathsf{Z}}\Prob_{z}\left(|\mathrm{er}_\uppi\bigl(\mathcal{A}^\varepsilon(Z_0,\dots,Z_{n-1})\bigr)-\mathrm{opt}_\uppi(\mathscr{H})|\ge5\varepsilon\right)
	\bigl(\upmu\mathcal{P}^n(\D z)-\uppi(\D z)\bigr)\right|\\
	&\le \delta+ \lVert\upmu\mathcal{P}^n(\D z)-\uppi(\D z)\rVert_\mathrm{TV}\\
	&\le \delta+C\upmu(\mathcal{V})\kappa^n.
	\end{align*}

\end{itemize}

In addition to the above assumptions assume that $\{Z_n\}_{n\ge0}$ is also reversible, i.e.  $$\int_A\mathcal{P}(z,B)\uppi(\D z)=\int_B\mathcal{P}(z,A)\uppi(\D z)\qquad\forall\, A,B\in\mathfrak{B}(\sfZ).$$ 
Then, according to \cite[Lemma 2.2]{Kontoyiannis-Meyn-2012} and \cite[Theorem 2.1]{Roberts-Rosenthal-1997},
 it admits a spectral gap in  $\mathrm{L}^2(\uppi)$, i.e.
 $$\lambda\df\sup_{f\in\mathrm{L}^2(\uppi),\, \lVert f\rVert_2=1}\lVert\mathcal{P}f-\uppi(f)\rVert_2<1.$$ 
Now,  from \cite[Theorem 3.3]{Miasojedow-2014}  it follows that 
$$\Prob_\uppi\bigl(|\hat{\mathrm{er}}_n(h)-\mathrm{er}_\uppi(h)|>\varepsilon\bigr)\le\E^{-2\frac{1-\lambda}{1+\lambda}\varepsilon^2n}\qquad \forall\, n\ge0.$$ 
From this, analogously as in the proof of \Cref{TM1.1}, we conclude that
 \begin{align*}\Prob_\uppi\left(\sup_{h\in\mathscr{H}}|\hat{\mathrm{er}}_n(h)-\mathrm{er}_\uppi(h)|>\varepsilon\right) &\le \mathcal{N}(\varepsilon/4\bar L,\mathscr{H},\uprho_\mathscr{H})\E^{-\frac{1-\lambda}{2(1+\lambda)}\varepsilon^2n}\qquad \forall\, n\ge0.\end{align*}
Let $\upmu(\D z)$ be any initial distribution of $\{Z_n\}_{n\ge0}$ such that $\upmu\ll\uppi$ and $ \frac{\D\upmu}{\D\uppi}\in\mathrm{L}^p(\uppi)$ for some $p\in(1,\infty]$. Then, 
	\begin{align*}\Prob_\upmu\left(\sup_{h\in\mathscr{H}}|\hat{\mathrm{er}}_n(h)-\mathrm{er}_\uppi(h)|>\varepsilon\right)&\le\bigl\lVert\frac{\D\upmu}{\D\uppi}\bigr\rVert_p\Prob_\uppi\left(\sup_{h\in\mathscr{H}}|\hat{\mathrm{er}}_n(h)-\mathrm{er}_\uppi(h)|>\varepsilon\right)^{1/q},\end{align*} where again $q\in[1,\infty)$ is such that $1/p+1/q=1$.
From this it follows that  for $\varepsilon>0$,  $\delta\in(0,1)$  and an $\varepsilon$-ASEM algorithm  $\mathcal{A}^\varepsilon$  for $\mathscr{H}$, 
we have that
$$\Prob_\upmu\left(|\mathrm{er}_\uppi\bigl(\mathcal{A}^\varepsilon(Z_0,\dots,Z_{n-1})\bigr)-\mathrm{opt}_\uppi(\mathscr{H})|<5\varepsilon\right)\ge 1-\delta$$ for all  $$n\ge \frac{2q(1+\lambda)}{(1-\lambda)}\log\left(\frac{\bigl\lVert\frac{\D\upmu}{\D\uppi}\bigr\rVert_p \mathcal{N}(\varepsilon/4\bar L,\mathscr{H},\uprho_\mathscr{H})^{1/q}}{\delta}\right).$$
Additionally, if $p=2$ from \cite[Proposition 1.5]{Kontoyiannis-Meyn-2012} we conclude
\begin{align*}
&\Prob_\upmu\left(|\mathrm{er}_\uppi\bigl(\mathcal{A}^\varepsilon(Z_n,\dots,Z_{2n-1})\bigr)-\mathrm{opt}_\uppi(\mathscr{H})|\ge5\varepsilon\right)\\
&\le \Prob_\uppi\left(|\mathrm{er}_\uppi\bigl(\mathcal{A}^\varepsilon(Z_0,\dots,Z_{n-1})\bigr)-\mathrm{opt}_\uppi(\mathscr{H})|\ge5\varepsilon\right)\\&\ \ \ +\bigl|\Prob_\upmu\left(|\mathrm{er}_\uppi\bigl(\mathcal{A}^\varepsilon(Z_n,\dots,Z_{2n-1})\bigr)-\mathrm{opt}_\uppi(\mathscr{H})|\ge5\varepsilon\right)\\&\hspace{1cm}-\Prob_\uppi\left(|\mathrm{er}_\uppi\bigl(\mathcal{A}^\varepsilon(Z_0,\dots,Z_{n-1})\bigr)-\mathrm{opt}_\upmu(\mathscr{H})|\ge5\varepsilon\right)\bigr|\\
&\le \delta+ \lVert\upmu\mathcal{P}^n(\D z)-\uppi(\D z)\rVert_\mathrm{TV}\\
&\le \delta+\frac{1}{2}\left(\int_{\mathsf{Z}}\left(1-\frac{\D\upmu}{\D\uppi}(z)\right)^2\uppi(\D z)\right)^{1/2}(1-\lambda)^n
\end{align*}
for all  $$n\ge \frac{2(1+\lambda)}{(1-\lambda)}\log\left(\frac{ \mathcal{N}(\varepsilon/4\bar L,\mathscr{H},\uprho_\mathscr{H})}{\delta}\right).$$


 
\section{Proof of the main results}
In this section, we prove \Cref{TM1.1,TM1.2}.

\subsection*{Proof of \Cref{TM1.1}} 
Fix $h\in\mathscr{H}$ and $z_0 \in \mathsf{Z}$. Then, (A3) and the boundedness of the metric $\rho$ on $\mathsf{Z}$, imply that 
$\lVert \ell^\uppi_h\rVert_\infty\le 2\lVert \ell_h\rVert_\infty \le2\sup_{z_1,z_2\in\sfZ}\uprho(z_1,z_2) + 2|\ell_h(z_0)| < \infty$, $\uppi(\ell^\uppi_h)=0$,
$\ell_h$ and $\ell^\uppi_h$ are Lipschitz continuous, and $\mathrm{Lip}(\ell^\uppi_h)=\mathrm{Lip}(\ell_h)$. 
Further, by employing Kantorovich-Rubinstein theorem, for any $z\in\sfZ$, from \cref{eq1.1} it follows that
$$\mathbb{E}_z\bigl[\ell_h^\uppi(Z_n)\bigr]\le \mathrm{Lip}(\ell^\uppi_h) 
\mathscr{W}\bigl(\mathcal{P}^n(z,\D\bar z),\uppi(\D \bar z)\bigr)\le C_1 \mathrm{Lip}(\ell_h) \E^{-C_2n} \le C_1 L \E^{-C_2n},$$
where in the last line we used assumption (A3), more precisely, we used that $|\ell_h(z_1) - \ell_h(z_2)| \le L \rho(z_1, z_2)$. Thus, the function \begin{equation}\label{eq:g}g(z)\df \sum_{n\ge0}\mathbb{E}_z\bigl[\ell_h^\uppi(Z_n)\bigr]\end{equation} is well defined. Furthermore, the function $g(z)$ satisfies
\begin{equation}\label{eq:norm_g_bound}
	\lVert g\rVert_\infty\le C_1L/(1-\E^{-C_2})
\end{equation}
and solves the equation
\begin{equation}\label{eq:g_solves}	g(z)-\mathbb{E}_z\bigl[g(Z_1)\bigr]=\ell_h^\uppi(z).
\end{equation}
Set now $W_i\df \ell_h(Z_i)$  and $S_n\df\sum_{i=0}^{n-1}W_i$ for $i\ge0$ and $n\ge1$. Thus, $$S_n-n\,\uppi(\ell_h)=n\,\hat{\mathrm{er}}_n(h)-n\,\mathrm{er}_\uppi(h).$$ Fix now arbitrary $\epsilon>0$. Analogously as in 
the proof of \cite[Theorem 2]{Glynn-Ormoneit-2002} we see that 
\begin{align*}\Prob_z\bigl(|\hat{\mathrm{er}}_n(h)-\mathrm{er}_\uppi(h)|>\epsilon\bigr)&=\Prob_z\bigl(|S_n-n\,\uppi(\ell_h)|>n\epsilon\bigr)\\&\le\E^{-(\epsilon n/\lVert g\rVert_\infty-2)^2/(2n)}\\&\le \E^{-\bigl(\frac{\epsilon n(1-\E^{-C_2})}{2C_1L}-2\bigr)^2/(2n)},\qquad \forall n\ge \frac{4C_1L}{\epsilon(1-\E^{-C_2})}.\end{align*}

Further, let $\mathscr{C}_\epsilon=\{h_1,\dots, h_{\mathcal{N}(\epsilon,\mathscr{H},\mathsf{d}_\mathscr{H})}\}$ be the $\epsilon$-cover of $\mathscr{H}$. Thus, for any $h\in\mathscr{H}$ there is $h_i\in \mathscr{C}_\epsilon$ such that $\mathsf{d}_\mathscr{H}(h,h_i)<\epsilon$. 
We now have,
\begin{align*}
\Prob_z\left(\sup_{h\in\mathscr{H}}|\hat{\mathrm{er}}_n(h)-\mathrm{er}_\uppi(h)|>4\bar L\epsilon\right)&\le\Prob_z\left(\bigcup_{i=1}^{\mathcal{N}(\epsilon,\mathscr{H},\mathsf{d}_\mathscr{H})}\left\{\sup_{h\in\{\bar h: \mathsf{d}_\mathscr{H}(\bar h,h_i)<\epsilon\}}|\hat{\mathrm{er}}_n(h)-\mathrm{er}_\uppi(h)|>4\bar L\epsilon\right\}\right)\\
&\le\sum_{i=1}^{\mathcal{N}(\epsilon,\mathscr{H},\mathsf{d}_\mathscr{H})}\Prob_z\left(\sup_{h\in\{\bar h: \mathsf{d}_\mathscr{H}(\bar h,h_i)<\epsilon\}}|\hat{\mathrm{er}}_n(h)-\mathrm{er}_\uppi(h)|>4\bar L\epsilon\right).
\end{align*}
For $h\in\{\bar h: \mathsf{d}_\mathscr{H}(\bar h,h_i)<\epsilon\}$, from (A2), it follows that $$|\hat{\mathrm{er}}_n(h)-\mathrm{er}_\uppi(h)-\hat{\mathrm{er}}_n(h_i)+\mathrm{er}_\uppi(h_i)|\le 2\bar L\uprho_\mathscr{H}(h,h_i)<2\bar L\epsilon.$$ Thus,
\begin{align*}\Prob_z\left(\sup_{h\in\{\bar h: \mathsf{d}_\mathscr{H}(\bar h,h_i)<\epsilon\}}|\hat{\mathrm{er}}_n(h)-\mathrm{er}_\uppi(h)|>4\bar L\epsilon\right)&\le \Prob_z\left(|\hat{\mathrm{er}}_n(h_i)-\mathrm{er}_\uppi(h_i)|>2\bar L\epsilon\right)\\
&\le\E^{-\bigl(\frac{\epsilon\bar L n(1-\E^{-C_2})}{C_1L}-2\bigr)^2/(2n)},\qquad \forall n\ge \frac{2C_1L}{\bar L\epsilon(1-\E^{-C_2})},
\end{align*}
which implies \begin{equation}\label{eq:unif}\Prob_z\left(\sup_{h\in\mathscr{H}}|\hat{\mathrm{er}}_n(h)-\mathrm{er}_\uppi(h)|>4\bar L\epsilon\right) \le \mathcal{N}(\epsilon,\mathscr{H},\mathsf{d}_\mathscr{H})\E^{-\bigl(\frac{\epsilon\bar L n(1-\E^{-C_2})}{C_1L}-2\bigr)^2/(2n)}\\
\end{equation} for all $n\ge 2C_1L/\bar L\epsilon(1-\E^{-C_2})$.

Finally, set $\varepsilon\df 4\bar L\epsilon$,   let $\mathcal{A}^\varepsilon$ be the $\varepsilon$-ASEM algorithm for $\mathscr{H}$, and  fix   $\delta\in(0,1)$. Then, 
$$\Prob_z\left(\sup_{h\in\mathscr{H}}|\hat{\mathrm{er}}_n(h)-\mathrm{er}_\uppi(h)|\le\varepsilon\right)\ge 1-\delta\qquad \forall n\ge n_1(\varepsilon,\delta).$$
Consequently,  for all $n\ge n_0(\varepsilon,\delta)$ it follows that with probability at least $1-\delta$,
 \begin{align*}\mathrm{er}_\uppi\bigl(\mathcal{A}^\varepsilon(Z_0,\dots,Z_{n-1})\bigr)&\le\mathrm{er}_\uppi(h_{(Z_0,\dots,Z_{n-1})})+\varepsilon/4\\
 &\le \hat{\mathrm{er}}_n(h_{(Z_0,\dots,Z_{n-1})})+5\varepsilon/4\\
 &\le\hat{\mathrm{er}}_n\bigl(\mathcal{A}^\varepsilon(Z_0,\dots,Z_{n-1})\bigr)+6\varepsilon/4\\
 &\le\inf_{h\in\mathscr{H}}\hat{\mathrm{er}}_n(h)+10\varepsilon/4\\
 &\le \hat{\mathrm{er}}_n(\bar h)+5\varepsilon/2\\
 &\le \hat{\mathrm{er}}_n(h_i)+11\varepsilon/4\\
 &\le \mathrm{er}_\uppi(h_i)+15\varepsilon/4\\
 &\le \mathrm{er}_\uppi(\bar h)+4\varepsilon\\
 &< \mathrm{opt}_\uppi(\mathscr{H})+5\varepsilon,\end{align*} where  $h_{(Z_0,\dots,Z_{n-1})}\in \mathscr{C}_{\varepsilon/4\bar L}$, $\bar h\in \mathscr{H}$ is such that $\mathrm{er}_\uppi(\bar h)<\mathrm{opt}_\uppi(\mathscr{H})+\varepsilon$, and  $h_i\in \mathscr{C}_{\varepsilon/4\bar L}$ such that $\mathsf{d}_\mathscr{H}(\bar h,h_i)<\varepsilon/4\bar L$.  \qed

 
\subsection*{Proof of \Cref{TM1.2}} 
Define
\begin{equation*}
	\tilde{\varepsilon} \df \sqrt{\frac{\varepsilon}{1 + \frac{1}{\alpha}}},
\end{equation*}
where $\varepsilon > 0$ is the one from the statement of this result. Using
\begin{equation*}
	n \ge  \frac{8C_1 L\sqrt{1+1/\alpha}}{\sqrt{m\varepsilon} (1 - \E^{-C_2})}=\frac{8C_1 L}{\sqrt{m}\tilde{\varepsilon} (1 - \E^{-C_2})},
\end{equation*}
together with \cref{eq:mM,eq:unif} we get
\begin{align*}
	\mathbb{P}_z \left( \sup_{h \in \mathscr{H}} \frac{\left| \hat{\rm er}_n(h) - \mathrm{er}_{\uppi}(h) \right|}{\sqrt{\mathrm{er}_{\uppi}(h)}} \ge \tilde{\varepsilon} \right)& \le \mathbb{P}_z \left( \sup_{h \in \mathscr{H}} \left| \hat{\rm er}_n(h) - \mathrm{er}_{\uppi}(h) \right| \ge \sqrt{m}\tilde{\varepsilon} \right)\\&\le  \mathcal{N} \left( \sqrt{m}\tilde{\varepsilon}/(4\bar{L}), \mathscr{H}, \mathsf{d}_{\mathscr{H}} \right) \E^{-\left(\frac{\sqrt{m}\tilde{\varepsilon}n(1-\E^{-C_2})}{4C_1L}-2\right)^2/(2n)}.
\end{align*}
Hence, we obviously have
\begin{equation}\begin{aligned}\label{eq:ineq_with_tilde}
&	\mathbb{P}_z \left(\exists\, h \in \mathscr{H} :  \frac{\mathrm{er}_{\uppi}(h) - \hat{\rm er}_n(h)}{\sqrt{\mathrm{er}_{\uppi}(h)}} \ge \tilde{\varepsilon} \right)\\  &\le  \mathcal{N} \left( \sqrt{m}\tilde{\varepsilon}/(4\bar{L}), \mathscr{H}, \mathsf{d}_{\mathscr{H}} \right) \E^{-\left(\frac{\sqrt{m}\tilde{\varepsilon}n(1-\E^{-C_2})}{4C_1L}-2\right)^2/(2n)}.
\end{aligned}\end{equation}
We now follow the proof of \cite[Theorem 5.8]{Anthony-Bartlett-Book-1999}. Suppose that
\begin{equation}\label{eq:key_eq_in_cor}
	\mathrm{er}_{\uppi}(h) - \hat{\rm er}_n(h) \le \tilde{\varepsilon} \sqrt{\mathrm{er}_{\uppi}(h)}.
\end{equation}
Then, for $\alpha > 0$ from the statement of this result (arbitrary, but fixed) we have two cases.

\medskip
\noindent
\textit{Case 1}: If $\mathrm{er}_{\uppi}(h) < (1 + 1/\alpha)^2 \tilde{\varepsilon}^2$ then \cref{eq:key_eq_in_cor} implies
\begin{equation*}
	\mathrm{er}_{\uppi}(h) \le \hat{\rm er}_n(h) + \tilde{\varepsilon}^2 \left( 1 + \frac{1}{\alpha} \right).
\end{equation*}

\medskip
\noindent
\textit{Case 2}: If $\mathrm{er}_{\uppi}(h) \ge (1 + 1/\alpha)^2 \tilde{\varepsilon}^2$ then $\tilde{\varepsilon} \le \sqrt{\mathrm{er}_{\uppi}(h)} / (1 + 1 / \alpha)$. Hence, from \cref{eq:key_eq_in_cor} it follows that
\begin{equation*}
	\mathrm{er}_{\uppi}(h) \le \hat{\rm er}_n(h) + \frac{\alpha}{\alpha + 1} \cdot \mathrm{er}_{\uppi}(h),
\end{equation*}
which is equivalent to
\begin{equation*}
	\mathrm{er}_{\uppi}(h) \le (1 + \alpha) \hat{\rm er}_n(h).
\end{equation*}
Therefore, we have just showed that \cref{eq:key_eq_in_cor} implies
\begin{equation*}
	\mathrm{er}_{\uppi}(h) \le (1 + \alpha) \hat{\rm er}_n(h) + \tilde{\varepsilon}^2 \left( 1 + \frac{1}{\alpha} \right) = (1 + \alpha) \hat{\rm er}_n(h) + \varepsilon.
\end{equation*}
Combining this with \cref{eq:ineq_with_tilde}, we get
\begin{align*}
	&\mathbb{P}_z
	 \left( \exists\, h \in \mathscr{H} : \mathrm{er}_{\uppi}(h) > (1 + \alpha) \hat{\rm er}_n(h) + \varepsilon  \right) \le \mathbb{P}_z \left( \exists\, h \in \mathscr{H} : \frac{\mathrm{er}_{\uppi}(h) - \hat{\rm er}_n(h)}{\sqrt{\mathrm{er}_{\uppi}(h)}} > \tilde{\varepsilon} \right) \\
	&  \le \mathcal{N} \left( \sqrt{m}\tilde{\varepsilon}/(4\bar{L}), \mathscr{H}, \mathsf{d}_{\mathscr{H}} \right) \E^{-\left(\frac{\sqrt{m}\tilde{\varepsilon}n(1-\E^{-C_2})}{4C_1L}-2\right)^2/(2n)} \\
	&  = \mathcal{N}\left( \frac{\sqrt{m\varepsilon}}{4\bar{L} \sqrt{1 + 1/\alpha}}, \mathscr{H}, \mathsf{d}_{\mathscr{H}} \right) \E^{-\left(\frac{\sqrt{m\varepsilon}n(1-\E^{-C_2})}{4C_1L\sqrt{1 + 1/\alpha}}-2\right)^2/(2n)},
\end{align*}
and this is clearly less than or equal to $\delta$ for
\begin{equation*}
	n \ge \frac{32C_1^2L^2(1+1/\alpha)}{m\varepsilon(1-\E^{-C_2})^2} \left( \frac{\sqrt{m\varepsilon}(1-\E^{-C_2})}{C_1 L\sqrt{1+1/\alpha}} + \ln \left( \frac{\mathcal{N}\left( \frac{\sqrt{m\varepsilon}}{4\bar{L} \sqrt{1 + 1/\alpha}}, \mathscr{H}, \mathsf{d}_{\mathscr{H}} \right)}{\delta} \right) \right),
\end{equation*}
which is exactly what we wanted to prove.
\qed

\section{Appendix}
 In this section, we prove the auxiliary result used in the construction of the example satisfying condition (A4).

\begin{lemma}\label{LM}
	 Let $\chain{X}$ be a Markov chain on $\mathsf{X}$ with transition kernel $\mathcal{P}_X(x,\D \bar x)$, and let 
	$f:\mathsf{X}\to f(\mathsf{X})\subseteq\mathsf{Y}$  be  measurable.
	The family of random variables $\{f(X_n)\}_{n\ge0}$ is a Markov chain on $\mathsf{Y}$ if 
	$\{f^{-1}(\{y\})\}_{y\in f(\mathsf{X})}$ forms a family of atoms for $\chain{X}$ on $f^{-1}(f(\mathsf{X}))$, i.e. for any $y\in f(\mathsf{X})$ and $x_1,x_2\in f^{-1}(\{y\})$, $$\mathcal{P}_X(x_1,f^{-1}(B))=\mathcal{P}_X(x_2,f^{-1}(B)).$$ In that case, transition kernel of  $\{f(X_n)\}_{n\ge0}$ is given by $$\mathcal{P}(y,B)=\mathcal{P}_X(x,f^{-1}(B))$$ for  $y\in f(\mathsf{X})$ and  $x\in f^{-1}(\{y\})$.
	\end{lemma}
\begin{proof} By assumption, $\chain{X}$ is defined on  $(\Omega,\mathcal{F},\Prob_x)_{x\in\mathsf{X}}$. Now, on  $(\Omega,\mathcal{F})$ define a family of probability measures $\{\bar{\mathbb{P}}_y\}_{y\in f(\mathsf{X})}$ by $\bar{\mathbb{P}}_y(f(X_n)\in B)\df\mathbb{P}_x(X_n\in f^{-1}(B))$ for $n\ge0$, $y\in f(\mathsf{X})$ and $x\in f^{-1}(\{y\})$. By assumption, $\{\bar{\mathbb{P}}_y\}_{y\in f(\mathsf{X})}$ is well defined, and the assertion now follows directly.
\end{proof}

\section*{Acknowledgements}
 Financial support through  \textit{Alexander von Humboldt Foundation} (No. HRV 1151902 HFST-E)  and \textit{Croatian Science Foundation} under project 8958 (for N.\ Sandri\'c), and the \textit{Croatian Science Foundation} under project 4197 (for S.\ \v Sebek) is gratefully acknowledged. We also thank the anonymous referee for helpful comments that have led to improvements of the presentation of the article.


\bibliographystyle{alpha}
\bibliography{References}

\end{document}